\documentclass[a4paper,draft,reqno,12pt]{amsart}
\usepackage[utf8]{inputenc}
\usepackage[T1,T2A]{fontenc}
\usepackage[english]{babel}
\usepackage{amsmath}
\usepackage{amssymb}
\usepackage{amscd}
\usepackage{amsthm}
\usepackage{euscript}
\usepackage[all]{xy}
\newtheorem{proposition}{Proposition}

\newtheorem{lemma}{Lemma}
\newtheorem{theorem}{Theorem}

\newtheorem{corollary}{Corollary}
\theoremstyle{definition}
\newtheorem{definition}{Definition}

\theoremstyle{remark}
\newtheorem {remark}{Remark}

\DeclareMathOperator{\Spec}{Spec}

\DeclareMathOperator{\Aut}{Aut}

\DeclareMathOperator{\Ker}{Ker}

\def\GG{{\mathbb G}}

\def\CC{{\mathbb C}}
\def\KK{{\mathbb K}}

\def\ZZ{{\mathbb Z}}

\def\QQ{{\mathbb Q}}

\def\AA{{\mathbb A}}

\def\VVV{\mathcal{V}}
\def\WWW{\mathcal{W}}

\def\OOO{\mathcal{O}}

\def\RRR{\mathfrak{R}}

\renewcommand{\phi}{\varphi}
\renewcommand{\ge}{\geqslant}
\renewcommand{\le}{\leqslant}

\begin{document}
\date{}
\title[Limit points and additive group actions]{Limit points and additive group actions}
\author{Ivan Arzhantsev}
\thanks{The research was supported by Russian Science Foundation, grant 19-11-00056}
\address{HSE University, Faculty of Computer Science, Pokrovsky Boulevard 11, Moscow, 109028 Russia}
\email{arjantsev@hse.ru}

\maketitle
\begin{abstract}
We show that an effective action of the one-dimensional torus $\GG_m$ on a normal affine algebraic variety $X$ can be extended to an effective action of a semi-direct product $\GG_m\rightthreetimes \GG_a$ with the same general orbit closures if and only if there is a divisor $D$ on $X$ that consists of $\GG_m$-fixed points. This result is applied to the study of orbits of the automorphism group $\Aut(X)$ on $X$. 
\end{abstract} 


\section{Introduction}
\label{sec1}

Let $X$ be an irreducible affine variety over an algebraically closed field $\KK$ of characteristic zero. Assume that $X$ is equipped with an effective action $\GG_m\times X\to X$ of the one-dimensional algebraic torus $\GG_m$. If the orbit $\GG_m\cdot x$ of some point $x\in X$ is non-closed, the complement $\overline{\GG_m\cdot x}\setminus \GG_m\cdot x$ consists of one point $x_0$ called the \emph{limit point} of the orbit $\GG_m\cdot x$. 

Limit points play an important role in the theory of algebraic transformation groups. Let $G$ be a reductive algebraic group acting linearly in a vector space $V$. A vector $v\in V$ is called \emph{nilpotent} if the closure of the orbit $G\cdot v$ contains zero. Equivalently, all homogeneous invariant polynomials of positive degree on $V$ vanish at $v$. The Hilbert-Mumford Criterion proved by Hilbert (1893) for the action of the unimodular group in a space of forms and by Mumford (1965) in the general situation claims that a vector $v\in V$ is nilpotent if and only if there is a one-dimensional torus $\GG_m\subseteq G$ such that $0\in\overline{\GG_m\cdot v}$; see e.g. \cite[Theorem~5.1]{PV}.

More generally, Richardson and Birkes proved that if a reductive algebraic group acts regularly on an affine variety $X$ and $x\in X$ is some point, then there is a one-dimensional torus $\GG_m\subseteq G$ such that the closure $\overline{\GG_m\cdot x}$ intersects the (unique) closed $G$-orbit in $\overline{G\cdot x}$; see \cite[Theorem~6.9]{PV}.

The study of limit points of a $\GG_m$-action on a smooth projective variety $X$ and the $\GG_m$-module structure on the tangent space of a limit point leads to the so-called Bialynicki-Birula decomposition of the variety $X$ into locally closed $\GG_m$-invariant subsets~\cite{BB}. This decomposition is used in many problems of Algebraic Geometry, Topology and Representation Theory. 

\smallskip

One may ask whether the limit point $x_0$ of a non-closed orbit $\GG_m\cdot x$ on an affine variety $X$ is equivalent to the point $x$ in the sense that $x_0$ can be sent to $x$ by an automorphism of~$X$. In general, this is not the case at least because the point $x_0$ can be singular while $x$ is smooth. At the same time, the cases when $x_0$ is equivalent to $x$ are of particular interest, and they are the object of study of this note. 

One way to guarantee that $x_0$ and $x$ are equivalent is to find an algebraic group $H$ acting on $X$ such that the orbit $H\cdot x$ contains both the orbit $\GG_m\cdot x$ and the point $x_0$. We are going to apply this idea in the case when $H$ is the additive group $\GG_a=(\KK,+)$. 

\begin{definition}
Let $X$ be an irreducible affine variety with an efffective action of the one-dimensional torus $\GG_m$. We say that a regular action $\GG_a\times X\to X$ is \emph{compatible} if the subgroup $\GG_m$ normalizes the subgroup $\GG_a$ in the automorphsim group $\Aut(X)$ and generic $\GG_a$-orbits on $X$ coincide with closures of generic $\GG_m$-orbits. 
\end{definition} 

Equivalently, to construct a compatible $\GG_a$-action on an affine $\GG_m$-variety $X$ is the same as to extend the $\GG_m$-action to an effective action of a semi-direct product $\GG_m\rightthreetimes \GG_a$ such that the generic orbit closures of the extended action coincides with generic orbit closures of the $\GG_m$-action. 

Let us recall that any orbit of a unipotent group action on an affine variety is closed ~\cite[Section~1.3]{PV}. So generic $\GG_a$-orbits on $X$ are closed curves isomorphic to the affine line.

\smallskip

Now we come to the main result of this note. 

\begin{theorem} \label{tmain}
Let $X$ be a normal affine variety with an effective action of the one-dimensional torus $\GG_m$. Then there exists a compatible $\GG_a$-action on $X$ if and only if the variety $X$ contains a prime divisor $D$ that is fixed by $\GG_m$ pointwise. 
\end{theorem} 

Theorem~\ref{tmain} allows for several applications. We say that an algebraic variety $X$ is \emph{almost homogeneous} if there is a regular action $G\times X\to X$ of an algebraic group $G$ with an open orbit. 

\begin{corollary} \label{cmain}
Let $X$ be a normal affine almost homogeneous variety and $D$ be a prime divisor on $X$. Assume that there is a non-trivial action of a one-dimensional torus $\GG_m$ on $X$ such that $D$ is fixed by $\GG_m$ pointwise.
Then the automorphism group $\Aut(X)$ acts on $X$ with an open orbit and this orbit intersects the divisor $D$. 
\end{corollary}

\medskip

Let us recall that an affine variety $X$ is called \emph{rigid}, if $X$ admits no non-trivial action of the group $\GG_a$. It is well known that $X$ is rigid if and only if the algebra $\KK[X]$ admits no non-zero locally nilpotent derivation \cite[Section~1.5.1]{Fr}. 

\begin{corollary} \label{rigid} 
Let $X$ be a normal affine variety with an effective action of the one-dimensional torus $\GG_m$. Assume that there is a prime divisor $D$ on $X$ that is fixed by $\GG_m$ pointwise. 
Then the variety $X$ is not rigid. In particular, if $\dim X\ge 2$ then the automorphism group $\Aut(X)$ is not a (finite dimensional) algebraic group. 
\end{corollary}

If an affine variety $X$ admits two actions of the torus $\GG_m$ that do not commute, then $X$ admits a non-trivial $\GG_a$-action; see \cite[Section~3]{FZ} and~\cite[Proof of Theorem~2.1]{AG}. This result plays a key role in the description of the automorphism group of a rigid variety~\cite{AG}. Theorem~\ref{tmain} is also a result of the same type: it guarantees the existence of a non-trivial $\GG_a$-action on $X$ provided $X$ admits a $\GG_m$-action of a specific form. 

\smallskip

We prove Theorem~\ref{tmain} and Corollaries~\ref{cmain}-\ref{rigid} in Section~\ref{sec2}. In Section~\ref{sec3} we give examples that illustrate Theorem~\ref{tmain}. In fact, the case of root subgroups on non-degenerate affine toric varieties \cite[Section~2]{AKZ} was the starting point of this work. We discuss this example in detail and present some generalizations. 

\smallskip

Let us mention some more researches related to the subject of this paper. In~\cite{FZ2}, a classification of locally nilpotent derivations on the algebra of regular functions on an affine surface with a $\CC^*$-action is given. In particular, \cite[Theorem~3.12]{FZ2} deals with locally nilpotent derivations, which we need in the proof of Theorem~\ref{tmain}, in the case of surfaces. A related result on existence of $\GG_a$-actions on affine varieties equipped with an action of the torus $\GG_m$ may be found in~\cite[Proposition~3]{GS}. Finally,  when the first version of this paper was written we have learned that a result similar to Theorem~\ref{tmain} appeared independently in the diploma paper by Valentin Losev (unpublished). 

\smallskip

\emph{Acknowledgments.}\ The author is grateful to Alvaro Liendo for a fruitful discussion, and to Sergey Gaifullin, Michail Zaidenberg, and two anonymous reviewers for many useful comments, suggestions and references to related results. 


\section{Proofs of main results}
\label{sec2}

Let $X$ be an irreducible algebraic variety equipped with a regular action $T\times X\to X$ of an algebraic torus $T$. Consider a regular action $\GG_a\times X\to X$ such that the image of $\GG_a$ in the automorphism group $\Aut(X)$ is normalized by the image of $T$. We say that such an action is \emph{vertical} if the induced action of the group $\GG_a$ on the field $\KK(X)^T$ of rational $T$-invariants is identical, and is \emph{horizontal} otherwise. Equivalently, a $T$-normalized $\GG_a$-action on $X$ is vertical if 
a generic $\GG_a$-orbit on $X$ is contained in the closure of a $T$-orbit on $X$ \cite[Section~1.2]{L1}; it follows from Rosenlicht's Theorem which claims that generic orbits of an action of an algebraic group are separated by rational invariants, see e.g. \cite[Theorem~2.3]{PV}. In turn, horizontal $\GG_a$-actions are characterized by the condition that generic $\GG_a$-orbits are transversal to generic $T$-orbits. In particular, compatible $\GG_a$-actions defined in the introduction are precisely vertical actions on affine varieties for $T=\GG_m$. 

Let us consider $\GG_m$-actions on affine varieties more systematically. There is a natural correspondence between effective $\GG_m$-actions on an affine variety $X$ and $\ZZ$-gradings on the algebra $A=\KK[X]$, i.e.
$$
A=\bigoplus_{i\in\ZZ} A_i, 
$$
such that the set of indices $i$ with $A_i\ne 0$ generates the group $\ZZ$. Let us recall that a homogeneous component $A_i$ consists of functions $f\in \KK[X]$ such that $f(t\cdot x)= t^if(x)$ for all $t\in\GG_m$ and all $x\in X$. Up to automorphism of the torus $\GG_m$ we may assume that there exists a positive $i$ with $A_i\ne 0$. 

A $\GG_m$-action on an affine variety $X$ is called \emph{elliptic} if  the only closed orbit is a $\GG_m$-fixed point $P$. In this case all $\GG_m$-orbits in $X$ contain $P$ in their closures. Equivalently, the algebra of invariants $\KK[X]^{\GG_m}$ coincides with the ground field $\KK$. In terms of gradings, elliptic actions correspond to positive gradings $A=\oplus_{i\ge 0} A_i$ of the algebra $A=\KK[X]$ with $A_0=\KK$. 

Oppositely, an effective $\GG_m$-action is \emph{hyperbolic} if generic $\GG_m$-orbits on $X$ are closed. This means that the grading  $A=\oplus_{i\in\ZZ} A_i$ includes non-zero components of both positive and negative degrees. Moreover, this condition is equivalent to existence of at least one closed orbit of positive dimension. In particular, if $A=\oplus_{i\ge 0} A_i$ then the closure of any one-dimensional $\GG_m$-orbit contains a limit point. 

Finally, we say that an effective $\GG_m$-action on an irreducible affine variety $X$ is \emph{parabolic} if there is an open subset $U\subseteq X$ such that all $\GG_m$-orbits in $U$ are non-closed in $X$ and their closures in $X$ do not intersect pairwise. 

\begin{lemma} \label{ulem}
Let $X$ be a normal affine variety equipped with an effective action of the torus~$\GG_m$. The following conditions are equivalent. 
\begin{enumerate}
\item[a)]
The action is parabolic.
\item[b)]
There is a prime divisor $D$ in $X$ that is fixed by $\GG_m$ pointwise. 
\item[c)]
The grading has the form $A=\oplus_{i\ge 0} A_i$ and the transcendence degree of the subalgebra $A_0=\KK[X]^{\GG_m}$ equals $\dim X-1$.  
\end{enumerate}
Under these conditions, the divisor $D$ coincides with the set of $\GG_m$-fixed points on $X$. 
\end{lemma} 

\begin{proof}
Let us consider the quotient morphism $\pi\colon X\to Y=X/\!/\GG_m:=\Spec\KK[X]^{\GG_m}$.  This morphism separates closed orbits and sends $\GG_m$-invariant closed subsets of $X$ to closed subsets of $Y$ \cite[Section~4.4]{PV}. Since every fiber of $\pi$ contains a unique closed orbit, all fibers are connected. By~\cite[Proposition~4]{Ar} we conclude that generic fibers of $\pi$ are irreducible.  

We show that a) implies b). Since the action is parabolic, generic orbit closures on $X$ are separated by the morphism $\pi$. This proves that generic fibers of $\pi$ are one-dimensional and $\dim Y=\dim X-1$. Moreover, the subvariety of $\GG_m$-fixed points on $X$ projects to $Y$ dominantly. We conclude that some irreducible component $D$ of this subvariety is a prime divisor. 

Now let us check that b) implies c). Since the $\GG_m$-action on $X$ is effective we have $\dim Y<\dim X$. On the other hand, the divisor $D$ projects to a closed subset of $Y$ and the projection is injective. This means the restriction of $\pi$ to $D$ is bijective.
In particular, all closed orbits on $X$ are fixed points and $\dim Y=\dim D=\dim X-1$. The first condition means that $A_i=0$ for all $i<0$ and the second one means that the transcendence degree of $A_0$ is $\dim X-1$. 

Let us show that c) implies a). Condition c) means that $\dim Y=\dim X-1$ and all closed orbits in $X$ are fixed points. This implies that generic orbits on $X$ are non-closed and their closures coincide with generic fibers of $\pi$. We conclude that 
generic orbit closures do not intersect pairwise, i.e., the action is parabolic.

Finally, since the restriction of $\pi$ to $D$ is bijective, every closed orbit on $X$ is a point in~$D$. 
\end{proof}

\begin{proof} [Proof of Theorem~\ref{tmain}]
Assume first that there is a compatible $\GG_a$-action on $X$. Then closures of generic $\GG_m$-orbits on $X$ coincide with $\GG_a$-orbits and so they do not intersect pairwise. This means that the $\GG_m$-action is parabolic and implication from a) to b) in 
Lemma~\ref{ulem} concludes the proof. 

Now let us assume that there is a prime divisor $D$ on $X$ that consists of $\GG_m$-fixed points. Then we have a grading $A=\oplus_{i\ge 0} A_i$ as in Lemma~\ref{ulem}, c). Let us show that all weight components $A_s$, $s>0$ are non-zero. Assume that some $A_s=0$. We localize the algebra $A$ by all non-zero elements of the subalgebra $B:=A_0$. The new algebra $C=A_{(B)}$ inherits the grading $C=\oplus_{i \ge 0} C_i$ from $A$, and $A_s=0$ implies $C_s=0$. By Lemma~\ref{ulem}, a), elements of $B$ separate 
generic $\GG_m$-orbits on $X$. This implies that the field $K:=\KK(X)^{\GG_m}$ of rational invariants coincides with the field of fractions of the algebra $B$ \cite[Lemma~2.1]{PV}. 

By construction, any two homogeneous elements in $C$ of the same degree are proportional over $K$. This shows that any homogeneous component in $C$ is at most one-dimensional over $K$. Since $C$ is a domain, we conclude that $C$ is a semigroup algebra
$K[\Gamma]$ for some subsemigroup $\Gamma$ in $(\ZZ_{\ge 0},+)$. By assumption, the algebra $A$ is integrally closed in its field of fractions. This implies that the algebra $C$ is integrally closed as well. But then the algebra $C$
is isomorphic to the polynomial algebra $K[T]$, where the element $T$ has degree $d>0$. Since the action of $\GG_m$ on $X$ is effective, we have $d=1$ and so $C_s\ne 0$ for all $s\ge 0$. 

We are going to construct a locally nilpotent derivation $\delta$ of the algebra $A$ that is homogeneous with respect to the grading and such that $\Ker\delta=B$. We start with the algebra $C=K[T]$ and the derivation $\delta'=\frac{\partial}{\partial T}$. Let $a_1,\ldots,a_m$ be homogeneous generators of the algebra~$A$. Then there is an element $b\in B$ such that all elements $b\delta'(a_1),\ldots,b\delta'(a_m)$ are in~$A$. Then the derivation $\delta=b\delta'$ is the required locally nilpotent derivation of the algebra~$A$. Consider the $\GG_a$-action of $X$ induced by the group $\{\text{exp}(s\delta) \, | \, s\in \KK\}$ of automorphisms of the algebra~$A$.  Such a $\GG_a$-subgroup is normalized by $\GG_m$ because the derivation $\delta$ is homogeneous, and the $\GG_a$-action is vertical because by construction $\delta$ annihilates all elements in $\KK(X)^{\GG_m}$. We conclude that the constructed $\GG_a$-action is compatible. 
\end{proof} 

\begin{remark}
The inverse implication in Theorem~\ref{tmain} should also follow from \cite[Corollary~2.8]{L1}. This approach is based on the technique of proper polyhedral divisors of Altmann and Hausen~\cite{AH}. In this note we prefer to use more elementary arguments given above.  
\end{remark}

\begin{proof} [Proof of Corollary~\ref{cmain}]
Since $X$ is almost homogeneous, there is a regular action $G\times X\to X$ of an algebraic group $G$ with an open orbit $\VVV$. Then the union of shifts $s\cdot\VVV$, where $s$ runs through all automorphisms of $X$, is an open orbit $\WWW$ of the group $\Aut(X)$ on $X$. By Theorem~\ref{tmain}, there is a compatible $\GG_a$-action on $X$. Since the divisor $D$ coincides with the set of $\GG_m$-fixed points on $X$, generic $\GG_a$-orbits intersect both $\VVV$ and $D$. This shows that generic points on $D$ lie in the same orbit of $\Aut(X)$ as $\VVV$. So the orbit $\WWW$ intersects the divisor $D$. 
\end{proof} 

\begin{proof} [Proof of Corollary~\ref{rigid}]
By Theorem~\ref{tmain}, there is a compatible $\GG_a$-action on $X$. So the variety $X$ is not rigid. It is well known that if an affine variety $X$ of dimension at least $2$ admits a $\GG_a$-action, then the automorphism group $\Aut(X)$ is not a (finite dimensional) algebraic group; see e.g. \cite[Proposition~7.5]{Kr}. Let us recall the arguments for this. Let $\delta$ be a locally nilpotent derivation on the algebra $\KK[X]$ that corresponds to the compatible $\GG_a$-action. The kernel $\Ker\delta$ is a subalgebra of transcendence degree $\dim X-1$; see~\cite[Principle~11]{Fr}. In particular, the kernel $\Ker\delta$ is an infinite dimensional $\KK$-vector space. For any finite dimensional subspace $V\subseteq\Ker\delta$ one can consider the subgroup $\{\text{exp}(f\delta)\, | \, f\in V\}$ of the automorphism group $\Aut(X)$. This is a commutative unipotent group of dimension $\dim V$. Since $\dim V$ is unbounded, all such subgroups can not be contained in a finite dimensional algebraic group. 
\end{proof} 

\begin{remark} \label{rem2}
Both Theorem~\ref{tmain} and Corollaries~\ref{cmain}-\ref{rigid} do not hold if the variety $X$ is non-normal: one may consider the cuspidal cubic $V(y^2-x^3)$ in $\AA^2$ with the torus action $t\cdot(x,y)=(t^2x,t^3y)$. 
\end{remark} 


\section{Examples and applications}
\label{sec3}

\subsection{Subtorus actions}
Let $N$ be the lattice of one-parameter subgroups of an algebraic torus $T$. Denote by $M$ the dual lattice of characters of $T$, and let $\langle\cdot,\cdot\rangle\colon N\times M \to\ZZ$ be the pairing between $N$ and $M$. We consider a regular action of $T$ on a normal affine variety $X$ and the corresponding grading
$$
A=\bigoplus_{u\in M} A_u
$$
of the affine algebra $A=\KK[X]$. Let $N_{\QQ}=N\otimes_{\ZZ}\QQ$ and $M_{\QQ}=M\otimes_{\ZZ}\QQ$ be the rational vector spaces generated by $N$ and $M$, respectively. We define the \emph{weight monoid} $M_X$ of the affine $T$-variety $X$
as the set of weights $u\in M$ such that $A_u\ne 0$, and the \emph{weight cone} $\omega(X)$ as the cone in $M_{\QQ}$ generated by the monoid $M_X$. Since the algebra $A$ is finitely generated, the monoid $M_X$ is finitely generated and the cone $\omega(X)$ is polyhedral. 

We say that a one-dimensional subtorus $\GG_m$ in $T$ is \emph{straightening} if the induced $\GG_m$-action on $X$ is parabolic. Let us recall that a \emph{facet} of a polyhedral cone is a proper face of maximal dimension. 

\begin{proposition} \label{rprop}
Let $X$ be a normal affine $T$-variety. Then straightening subtori in $T$ are in bijection with some facets of the cone $\omega(X)$. In particular, the number of straightening subtori in $T$ is at most finite. 
\end{proposition}

\begin{proof}
Observe that a one-parameter subgroup $\GG_m$ in $T$ is nothing but an integer linear function $l$ on the space $M_{\QQ}$. In these terms, the grading on $A$ corresponding to the $\GG_m$-action on $X$ is 
\begin{equation} \label{grading} 
A=\bigoplus_{i\in \ZZ} A_i, \quad \text{where} \quad A_i=\bigoplus_{l(u)=i} A_u.   
\end{equation}
We conclude that such a grading contains no negative component if and only if the intersection of $\omega(X)$ with the hyperplane $\{l=0\}$ is a face of $\omega(X)$. For any face $\tau$ of the cone $\omega(X)$ we denote by $A(\tau)$ the subalgebra
$\oplus_{u\in\tau}A_u$. If a face $\tau$ is a facet of a face $\lambda$, then $A(\tau)\subseteq A(\lambda)$ and the transcendence degree of $A(\lambda)$ exceeds the transcendence degree of $A(\tau)$ by at least one. So for grading~(\ref{grading})
without negative components we may have that the transcendence degree of the subalgebra $A_0$ is $\dim X-1$ only if $A_0$ has the form $A(\lambda)$ for some facet $\lambda$ of the cone $\omega(X)$. Lemma~\ref{ulem}, c) implies that straightening one-parameter subgroups in $T$ correspond to some facets of $\omega(X)$. 
\end{proof} 

Now we come to an important particular case of subtorus actions.

\smallskip

\subsection{Affine toric varieties} Assume that a torus $T$ acts effectively on a normal affine variety $X$ with an open orbit, i.e., $X$ is an affine toric $T$-variety. Equivalently, we have $\dim A_u=1$ for any $u\in M_X$. Moreover, in this case the algebra $A$ is the semigroup algebra of the semigroup $M_X=\omega(X)\cap M$. We refer to~\cite{CLS,Fu,Oda} for a general theory of toric varieties. 

Under these assumptions, the transcendence degree of every subalgebra $A(\lambda)$, where $\lambda$ is a facet of the cone $\omega(X)$, equals $\dim X-1$. So we have a bijection between straightening subtori in $T$ and facets of $\omega(X)$. 

Let $\sigma(X):=\omega(X)^{\vee}$ be the cone dual to $\omega(X)$ in the dual space $N_{\QQ}$. Then the primitive lattice vectors $p_1,\ldots,p_m$ on the rays of $\sigma(X)$ are precisely the straightening subgroups in~$T$. It is well known that the rays of $\sigma(X)$ are in bijection with $T$-invariant prime divisors $D_1,\ldots,D_m$ on $X$; see, e.g.,~\cite[Theorem~3.2.6]{CLS}. Under this correspondence, the divisor $D_i$ is the set of fixed points of the subgroup $p_i$ \cite[Proposition~3.2.2]{CLS}. 

We say that a vector $e\in M$ is a \emph{Demazure root} of the cone $\sigma(X)$ if there is $1\le s\le m$ such that $\langle p_s,e\rangle=-1$ and $\langle p_i,e\rangle\ge 0$ for all $i\ne s$. So, the set $\RRR$ of all Demazure roots of the cone $\sigma(X)$ is a disjoint union of subsets $\RRR_1,\ldots,\RRR_m$ indexed by the corresponding index $s$. One can easily check that if $\dim X\ge 2$ then every set $\RRR_s$ is infinite. 

It is known that Demazure roots of $\sigma(X)$ are in bijection with $\GG_a$-subgroups in $\Aut(X)$ normalized by $T$~\cite[Theorem~1.6]{L1}. Moreover, the $\GG_a$-subgroup $H_e$ corresponding to a root $e\in\RRR_s$ acts on $X$ compatibly with the $\GG_m$-subgroup given by the vector $p_s\in N$~\cite[Proposition~2.1]{AKZ}. Summarizing these results, we obtain the following proposition.

\begin{proposition}
Let $X$ be an affine toric $T$-variety. Then straightening $\GG_m$-subgroups in $T$ are in bijection with $T$-invariant prime divisors, and every $T$-invariant prime divisor on $X$ is the set of $\GG_m$-fixed points for the corresponding straightening subtorus. 
\end{proposition}

This shows that an open orbit $\OOO$ of the group $\Aut(X)$ on $X$ intersects any $T$-invariant prime divisor. In fact, it is known that $\OOO$ coincides with the regular locus of $X$, and if $X$ is non-degenerate and $\dim X\ge 2$, then the group $\Aut(X)$ acts on $\OOO$ infinitely transitively \cite[Theorem~2.1]{AKZ}.

\begin{remark}
On a non-normal affine toric variety $X$, an open $\Aut(X)$-orbit may not intersect any $T$-invariant prime divisor on $X$. This is the case, for example, for the Neil's parabola from Remark~\ref{rem2}; for a general result on this effect, see~\cite[Theorem~2]{BG}.
\end{remark}

\begin{remark}
A description of all orbits of the group $\Aut(X)$ on an affine toric variety $X$ is obtained in~\cite[Theorem~5.1 and Proposition~5.6]{AB}.
\end{remark}

One should note that if we go from the toric case to an arbitrary torus action on a~normal affine variety, the situation may change completely. Let us consider an example of a torus action with generic orbits of codimension one. Fix positive integers $n_0,n_1,n_2$ with $n=n_0+n_1+n_2$. For each $i=0,1,2$, fix a tuple $l_i\in\ZZ_{>0}^{n_i}$ with $l_{i1}n_i>1$ and define a monomial
$$
T_i^{l_i}:=T_{i1}^{l_{i1}}\ldots T_{in_i}^{l_{in_i}}\in\KK[T_{ij}; \ i=0,1,2, \, j=1,\ldots,n_i].
$$
By a trinomial we mean a polynomial of the form $f=T_0^{l_0}+T_1^{l_1}+T_2^{l_2}$. A \emph{trinomial hypersurface} $X$ is the zero set $\{f=0\}$ in the affine space $\AA^n$. One can check that $X$ is a~normal affine variety of dimension $n-1$ and the variety $X$ carries a natural effective action of an $(n-2)$-dimensional torus $T$; see~\cite{HW}. By~\cite[Lemma~2]{Ar2}, there is no vertical $\GG_a$-action on $X$. In particular, there is no $T$-normalized $\GG_a$-action on $X$ that is compatible with a one-dimensional subtorus in $T$. 

\subsection{Homogeneous fiber spaces} We are going to apply the results discussed in the previous subsection to a wider class of almost homogeneous varieties. Let $G$ be a linear algebraic group and $T$ be a subtorus in $G$. We take an affine toric $T$-variety $Y$ and consider the homogeneous fiber space $X=G*_TY$, which is the quotient of the direct product $G\times Y$ by the $T$-action $t\cdot(g,y)=(gt^{-1},ty)$; see~\cite[Section~4.8]{PV} for details. The variety $G*_TY$ projects to the homogeneous space $G/T$, and all fibers of the projection are isomorphic to~$Y$. The $G$-action on $G\times Y$ by left multiplication on the first factor induces the $G$-action on $X$ with an open orbit. Moreover, the $T$-action on $G\times Y$ via the action on the second factor descents to a $T$-action on $X$. 

\begin{proposition} \label{pr1}
Let $G$ be a connected linear algebraic group, $T$ be a subtorus in $G$, and $Y$ be an affine toric $T$-variety. Consider the homogeneous fiber space $X=G*_TY$. Then an open orbit of the group $\Aut(X)$ on $X$ intersects every $G$-invariant prime divisor on $X$. 
\end{proposition}

\begin{proof}
Every $G$-invariant prime divisor on $X$ has the form $G*_TD$, where $D$ is a $T$-invariant prime divisor on $Y$. The action of a straightening $\GG_m$-subgroup on $Y$ that fixes $D$ pointwise extends to $X$ and has the divisor $G*_TD$ as the set of fixed points. So the assertion follows from Corollary~\ref{cmain}.
\end{proof}

\subsection{Affine varieties with two orbits} Finally, let us show that if one adds a homogeneous divisor to a homogeneous space then the extended space is still homogeneous.  

\begin{proposition} \label{pr2}
Let $G$ be a reductive group and $X$ be a normal affine $G$-variety consisting of two $G$-orbits $\OOO_1$, $\OOO_2$ with $\dim\OOO_1=\dim\OOO_2+1$. Then the variety $X$ is $\Aut(X)$-homogeneous. 
\end{proposition}

\begin{proof}
Since a normal variety is smooth in codimension one, the variety $X$ is smooth. Let $K$ be the stabilizer of a point in the orbit $\OOO_2$. By~\cite[Theorem~6.7]{PV}, the variety $X$ is isomorphic to $G*_KV$, where $V$ is a $K$-module. The condition $\dim\OOO_1=\dim\OOO_2+1$ implies that $V$ is one-dimensional. So $K$ acts on $V$ by some character. The action of $\GG_m$ on $G\times V$ by scalar multiplication on the second factor descents to $G*_KV$, and the divisor $\OOO_2$ consists of $\GG_m$-fixed points. So the assertion follows from Corollary~\ref{cmain}.
\end{proof}

\begin{remark}
If the group $G$ is semisimple, a more general version of  Proposition~\ref{pr2} is given in~\cite[Theorem~5.6]{AFKKZ}.
\end{remark}



\end{document}